\documentclass[11pt,a4paper]{amsart}
\usepackage{amsfonts,amsmath,amssymb,amsthm}
\usepackage{times}
\usepackage{color}
\usepackage{dsfont}
\usepackage[noadjust]{cite}
\usepackage{graphicx}

\numberwithin{equation}{section}

\renewcommand{\Re}{{\ensuremath{\mathrm{Re\,}}}}

\DeclareSymbolFont{SY}{U}{psy}{m}{n}
\DeclareMathSymbol{\emptyset}{\mathord}{SY}{'306}

\DeclareMathOperator{\Ran}{Ran}
\DeclareMathOperator{\Ker}{Ker}
\DeclareMathOperator{\Dom}{Dom}

\DeclareMathOperator{\spec}{spec}

\newcommand{\norm}[1]{\|#1\|}
\newcommand{\abs}[1]{|#1|}

\newcommand{\dav}{\mathfrak D}

\newcommand{\C}{\mathbb{C}}
\newcommand{\R}{\mathbb{R}}
\newcommand{\N}{\mathbb{N}}

\newcommand{\EE}{\mathsf{E}}

\newcommand{\cC}{{\mathcal C}}
\newcommand{\cD}{{\mathcal D}}
\newcommand{\cG}{{\mathcal G}}
\newcommand{\cH}{{\mathcal H}}
\newcommand{\cL}{{\mathcal L}}
\newcommand{\cN}{{\mathcal N}}
\newcommand{\cS}{{\mathcal S}}
\newcommand{\cT}{{\mathcal T}}
\newcommand{\cU}{{\mathcal U}}
\newcommand{\cV}{{\mathcal V}}

\newtheorem{introtheorem}{Theorem}{\bf}{\it}

\newtheorem{theorem}{Theorem}[section]{\bf}{\it}
\newtheorem{proposition}[theorem]{Proposition}{\bf}{\it}
\newtheorem{corollary}[theorem]{Corollary}{\bf}{\it}

\newtheorem{lemma}[theorem]{Lemma}{\bf}{\it}
\newtheorem{remark}[theorem]{Remark}{\it}{\rm}
\newtheorem{hypothesis}[theorem]{Hypothesis}{\bf}{\it}
\newtheorem{example}[theorem]{Example}{\it}{\rm}

\renewcommand{\theenumi}{\alph{enumi}}

\title[On graph subspaces]{On invariant graph subspaces}

\author[K.\ A.\ Makarov]{Konstantin A.\ Makarov}
\address{K.~A.~Makarov,
Department of Mathematics, University of Missouri,
Columbia, MO 65211, USA}
\email{makarovk@missouri.edu}

\author[S.\ Schmitz]{Stephan Schmitz}
\address{S.~Schmitz,
Department of Mathematics, University of Missouri,
Columbia, MO 65211, USA}
\email{schmitzst@missouri.edu}

\author[A.\ Seelmann]{Albrecht Seelmann}
\address{A.~Seelmann,
FB 08 - Institut f\"{u}r Mathematik, Johannes Gutenberg-Universi\-t\"{a}t Mainz,
Staudinger Weg 9, D-55099 Mainz, Germany}
\email{seelmann@mathematik.uni-mainz.de}

\begin{document}

\begin{abstract} 
In this paper we discuss the problem of decomposition  for unbounded  $2\times 2$ operator matrices by a pair of complementary
invariant graph subspaces. Under mild additional assumptions, we show that such a pair of subspaces decomposes the operator matrix
if and only if its domain is invariant for the angular operators associated with the graphs. As a byproduct of our considerations,
we suggest a new block diagonalization procedure that resolves related domain issues. In the case when only a single invariant
graph subspace is available, we obtain block triangular representations for the operator matrices.
 
\end{abstract}

\maketitle

\section{Introduction}

In the present work, we consider linear operators $B$ on a Hilbert space $\cH$ represented by $2\times 2$ operator matrices of the
form
\begin{equation}\label{eq:blockOp}
 B = \begin{pmatrix} A_0 & W_1\\ W_0 & A_1 \end{pmatrix} = \begin{pmatrix} A_0 & 0\\0 & A_1 \end{pmatrix} +
     \begin{pmatrix}0 & W_1\\ W_0 & 0\end{pmatrix}=:A+V
\end{equation}
with respect to a given orthogonal decomposition $\cH=\cH_0\oplus\cH_1$. In particular, if $B$ is unbounded, the operator matrix is
defined on its natural domain
\[
 \Dom(B) = \Dom(A+V) = \Dom(A) \cap \Dom(V)\,.
\]

Suppose that there is a pair of closed complementary subspaces $\cG_0$ and $\cG_1$ of $\cH$, that is,
\[
 \cG_0 +\cG_1=\cH \quad\text{ and }\quad \cG_0\cap\cG_1=\{0\}\,,
\]
such that both are invariant for $B$. If, in addition, this pair of subspaces \emph{decomposes} the operator $B$ in the sense that
the domain $\Dom(B)$ splits as
\begin{equation}\label{eq:intro:domSplit}
 \Dom(B) = \bigl(\Dom(B)\cap\cG_0\bigr) + \bigl(\Dom(B)\cap\cG_1\bigr)\,,
\end{equation}
then the operator $B$ can be represented as the direct sum of its parts $B|_{\cG_0}$ and $B|_{\cG_1}$, the restrictions of $B$ to
the subspaces $\cG_0$ and $\cG_1$, respectively.

Note that, in the case of unbounded operators $B$, the splitting property \eqref{eq:intro:domSplit} of the domain is not
self-evident, even if $B$ is self-adjoint; see \cite[Example 1.8]{Schm12} for a counterexample. For further discussion of the
notion of the decomposition of an operator by a pair of complementary subspaces we refer,  e.g., to \cite[Sec.\ III.5.6]{Kato66}.

Throughout this work, we are mostly interested in the particular case of complementary \emph{graph subspaces}
\begin{equation}\label{G1}
 \cG_0=\{f\oplus X_0f \mid f\in\cH_0\}=:\cG(\cH_0,X_0)
\end{equation}
and
\begin{equation}\label{G2}
 \cG_1=\{X_1g\oplus g \mid g\in\cH_1\}=:\cG(\cH_1,X_1)\, ,
\end{equation}
associated with bounded linear operators $X_0\colon\cH_0\to\cH_1$ and $X_1\colon\cH_1\to\cH_0$, respectively.

It should be noted that such invariant graph subspaces do not always exist, even if the operator matrix $B$ is bounded and
self-adjoint, see  \cite[Sec.\ 4]{KMM03} for a counterexample. On the other hand, in the current setting, considerable efforts
have been undertaken to show their existence in particular cases
\cite{AL95,ADL01,AMM03,Cue12,H69,KMM02,KMM2004,KMM2005,LT97,Mot95,MoSe06,Shk07,TW14,W11}; see also Theorem \ref{thm:subordinated}
below.

Given a pair of complementary invariant graph subspaces \eqref{G1} and \eqref{G2}, the consideration of the intersections
$\Dom(B)\cap\cG_j$, $j=0,1$, especially if the domain splitting \eqref{eq:intro:domSplit} is concerned, requires some extra
knowledge on mapping properties of the angular operators $X_j$. For instance, writing $\Dom(B)=\dav_0\oplus\dav_1$ with respect to
$\cH=\cH_0\oplus\cH_1$, it is a natural question whether or not the identities
\begin{equation}\label{eq:regular0}
 \Dom(B) \cap \cG(\cH_0,X_0) = \{f\oplus X_0f \mid f\in\dav_0\}
\end{equation}
and
\begin{equation}\label{eq:regular1}
 \Dom(B) \cap \cG(\cH_1,X_1) = \{X_1g\oplus g \mid g\in\dav_1\}
\end{equation}
hold. In this context, it is useful to identify the operators $X_0$ and $X_1$ with their corresponding trivial continuation to the
whole Hilbert space $\cH$, for which we keep the same notation. Upon this identification, the identities \eqref{eq:regular0} and
\eqref{eq:regular1} simply mean that the domain $\Dom(B)$ is invariant for $X_0$ and $X_1$, respectively. That is, the condition
$x=x_0\oplus x_1\in\Dom(B)$ implies that $X_jx=X_jx_j$ is contained in $\Dom(B)$, $j=0,1$.

It turns out that the invariance of $\Dom(B)$ for $X_0$ and $X_1$ above is closely related to the splitting property
\eqref{eq:intro:domSplit} for the complementary invariant graph subspaces $\cG_0$ and $\cG_1$. In fact, the following main result
of the present paper shows that these two requirements are equivalent under mild additional assumptions:

\begin{introtheorem}\label{thm:main}
 Let $\cG_0=\cG(\cH_0,X_0)$ and $\cG_1=\cG(\cH_1,X_1)$ be complementary graph subspaces associated with bounded linear operators
 $X_0\colon\cH_0\to\cH_1$ and $X_1\colon\cH_1\to\cH_0$, respectively. Suppose that the subspaces $\cG_0$ and $\cG_1$ both are
 invariant for the operator matrix $B$ in \eqref{eq:blockOp}. Assume, in addition, that the operators $B=A+V$ and $A-YV$ with
 \[
  Y := \begin{pmatrix} 0 & X_1\\ X_0 & 0 \end{pmatrix}
 \]
 are closed and have a common point $\lambda$ in their resolvent sets.

 Then, the following are equivalent:
 \begin{itemize}
 \item[i)]  the domain $\Dom(B)$ splits as in  \eqref{eq:intro:domSplit};
 \item[ii)]  the graph  subspaces $\cG_0$ and $\cG_1$  are  invariant for $(B-\lambda)^{-1}$;
 \item[iii)]  the domain $\Dom(B)$ is invariant for the angular operators $X_0$ and $X_1$.
 \end{itemize}
\end{introtheorem}

It should be noted that in \cite{Shk07}, the invariance of $\Dom(B)$ for the angular operator(s) has been incorporated into the
notion of invariance for graph subspaces. This, however, deviates from the standard notion of invariance for general subspaces
(see, e.g., \cite[Definition 2.9.11]{Tre08}). Moreover, the invariance of $\Dom(B)$ for the angular operators $X_0$ and $X_1$
without additional hypotheses is far from being obvious, even if the requirements i) and/or ii) are satisfied.

Regarding the proof of Theorem \ref{thm:main}, we first remark that the equivalence between i) and ii) is essentially well known,
even in a more general context, see, e.g., \cite[Remark 2.3 and Lemma 2.4]{TW14}; see also Lemma \ref{lem:resInv} below. Thus, the
proof of Theorem \ref{thm:main} reduces to the justification of either of the equivalences i) $\Leftrightarrow$ iii) or ii)
$\Leftrightarrow$ iii).

In this paper we justify these equivalences independently, thereby providing two alternative proofs of Theorem \ref{thm:main},
which shed some light on different aspects of the problem.

Our first proof of Theorem \ref{thm:main} establishes the equivalence between i) and iii). Here, the reasoning is, in essence,
based on the observation that under either of the conditions i) and iii) the operator matrix $B=A+V$ admits the block
diagonalization
\begin{equation}\label{mainfac}
 (I_\cH-Y)(A+V) (I_\cH-Y)^{-1} = A-YV = \begin{pmatrix} A_0-X_1W_0 & 0\\ 0 & A_1-X_0W_1 \end{pmatrix},
\end{equation}
see the discussion after Remark \ref{rem:closed} below.

Note that the concept of block diagonalization for operator matrices with unbounded entries has already been widely discussed in
the literature, see, e.g., \cite{AMM03,Cue12,LT97,Tre08}. However, the general statement of Theorem \ref{thm:main}, as well as the
similarity relation in the particular form \eqref{mainfac}, seems to be new; a detailed discussion on old and new results in this
area can be found in Remarks \ref{rem:alternative} and \ref{rem:comparisonLT} below.

In our second, independent, proof of Theorem \ref{thm:main}, we directly show instead that ii) and iii) are equivalent. In fact,
this is done by dealing with the graph subspaces $\cG_0$ and $\cG_1$ separately:
The subspace $\cG_j$ is invariant for $(A+V-\lambda)^{-1}$ if and only if $\Dom(B)$ is invariant for the angular operator $X_j$,
$j=0,1$, see Theorem \ref{thm:single} below. The proof of the latter rests to some extent on the Schur block triangular
decomposition, which, in the particular case of $j=0$, has the upper triangular form (see Lemma \ref{lem:triDiag} below)
\begin{equation}\label{eq:triDiagintro}
 \begin{pmatrix} I_{\cH_0} & 0\\ -X_0 & I_{\cH_1} \end{pmatrix} (A+V) \begin{pmatrix} I_{\cH_0} & 0\\ X_0 & I_{\cH_1} \end{pmatrix}
 =\begin{pmatrix} (A_0+W_1X_0)|_{\dav_0} & W_1\\ 0 & A_1-X_0W_1 \end{pmatrix},
\end{equation}
provided that $\Dom(B)$ is invariant for $X_0$.

In both approaches, our considerations rely on a detailed study of mapping properties of the angular operators $X_j$, $j=0,1$,
which are solutions to the associated operator Riccati equations, see eqs.\ \eqref{eq:RiccatiX0} and \eqref{eq:RiccatiX1} below. It
is well known that these equations play an important role in the search for invariant graph subspaces in general. In the context of
the present paper, the Riccati equations eventually yield the block diagonalization \eqref{mainfac} and the block triangular
representation \eqref{eq:triDiagintro}. For further discussion of operator Riccati equations in perturbation theory for block
operator matrices, we refer to \cite{AMM03}, the monograph \cite{Tre08}, and references therein.

In Theorem \ref{thm:main}, the condition of $A+V$ and $A-YV$ to have a common point in their resolvent sets is natural in the sense
that their resolvent sets will eventually agree by \eqref{mainfac}. In fact, our first proof of Theorem \ref{thm:main} shows that
the block diagonalization \eqref{mainfac} is available as soon as the statements in assertions i) and iii) of Theorem
\ref{thm:main} as such hold simultaneously, see Proposition \ref{prop:extRel} below. Thus, unless the resolvent set of $A+V$ is
empty, the condition of intersecting resolvent sets is not only sufficient but also necessary for the claimed equivalence to hold.
However, it is unclear whether the operators $A+V$ and $A-YV$ always have a common point in their resolvent sets. It therefore
remains an open problem whether the domain splitting \eqref{eq:intro:domSplit} and the invariance of $\Dom(B)$ for the angular
operators $X_0$ and $X_1$ are in general logically independent or not.

At this point, it should be noted that the resolvent sets of $A+V$ and $A-YV$ automatically intersect if, say, the diagonal part
$A$ is self-adjoint and the off-diagonal part $V$ is small in some sense, e.g.\ bounded or relatively bounded with sufficiently
small $A$-bound, see Corollary \ref{cor:Neumann}\,(b) below. In this regard, Theorem \ref{thm:main} can be interpreted as an
extension of \cite[Lemma 5.3]{AMM03}, where that case with bounded symmetric $V$ was discussed.

In our considerations in Sections \ref{sec:app} and \ref{sec:Example} to guarantee intersecting resolvents sets we restrict ourselves
to the diagonally dominant case, that is, to the case where $V$ is relatively bounded with respect to $A$. Nevertheless, or results
might also be useful for block diagonalization of some classes of off-diagonally dominant matrices, in particular, Dirac operators
\cite{Cue12}, \cite{Th92}. We briefly discuss a relevant application in Solid State Physics in
Example \ref{exgraph} at the end of  Section \ref{sec:Example}. For a comprehensive exposition of other applications in
mathematical physics we refer to \cite[Chapter 3]{Tre08} and references therein.

The paper is organized as follows.

In Section \ref{sec:invSub}, we collect some preliminary facts on pairs of invariant (graph) subspaces. In particular, we provide a
proof of the equivalence between assertions i) and ii) of Theorem \ref{thm:main}.

The equivalence between i) and iii) in Theorem \ref{thm:main} is shown in Section \ref{sec:mainThm}, where also the block
diagonalization formula \eqref{mainfac} is derived. Furthermore, this block diagonalization is compared to previously known results
in the literature.

In Section \ref{sec:general}, we  show the equivalence between ii) and iii), thus providing the second independent proof of the
theorem, and establish the block triangular decomposition \eqref{eq:triDiagintro}.

Section \ref{sec:app} is devoted to relatively bounded off-diagonal perturbations $V$ of a closed diagonal operator matrix $A$. We
discuss sufficient conditions on $A$ and $V$ that ensure the existence of a common point in the resolvent sets of $A+V$ and $A-YV$,
so that Theorem \ref{thm:main} can be applied.

In the final Section \ref{sec:Example}, as an example of our considerations, we block diagonalize a self-adjoint operator matrix
$B=A+V$ for which the spectra of the diagonal entries are subordinated, cf.\ \cite[Corollary 3.2]{LT97}.
 
Some words about notation:

The domain of a linear operator $K$ is denoted by $\Dom(K)$, its range by $\Ran(K)$, and its kernel by $\Ker(K)$. The restriction
of $K$ to a given subset $\cC$ of $\Dom(K)$ is written as $K|_{\cC}$.

Given another linear operator $L$, we write the extension relation $K\subset L$ (or $L\supset K$) if $L$ extends $K$, that is, if
$\Dom(K)\subset\Dom(L)$ and $Kx=Lx$ for $x\in\Dom(K)$. The operator equality $K=L$ means that $K\subset L$ and $K\supset L$.

We write $\rho(K)$ for the resolvent set of a closed operator $K$ on a Hilbert space, and $K^*$ stands for the adjoint operator of
$K$ if $K$ is densely defined. The identity operator on a Hilbert space $\cH$ is written as $I_{\cH}$. Multiples $\lambda I_{\cH}$
of the identity are abbreviated by $\lambda$. Finally, the inner product and the associated norm on $\cH$ are denoted by
$\langle\cdot,\cdot\rangle_{\cH}$ and $\norm{\cdot}_{\cH}$, respectively, where the subscript $\cH$ is usually omitted.

\subsection*{Acknowledgements.}
Individual parts of the material presented in this work are contained in the Ph.D.\ theses \cite{SchmDiss} and \cite{SeelDiss} by
the authors S.\ Schmitz and A.\ Seelmann, respectively.

The authors would like to thank Vadim Kostrykin for helpful discussions on the topic.\\
K.~A.~Makarov is indebted to the Institute for Mathematics for its kind hospitality during his one month stay at the Johannes
Gutenberg-Universit\"at Mainz in the Summer of 2014. The work of K.\ A.\ Makarov has been supported in part by the Deutsche
Forschungsgemeinschaft, grant KO 2936/7-1.

\section{Invariant subspaces}\label{sec:invSub}

In this first section, we introduce the basic notions used throughout the paper and discuss preliminary facts on pairs of invariant
(graph) subspaces. In particular, we reproduce the proof of the equivalence between the assertions i) and ii) of Theorem
\ref{thm:main}.

Let $B$ be a linear operator on a Hilbert space $\cH$. A subspace $\cU\subset\cH$ is called \emph{invariant} for $B$ if
\[
 Bx\in\cU \quad\text{ for all }\quad x\in\Dom(B)\cap\cU\,.
\]
A pair of complementary subspaces $\cU,\cV\subset\cH$, that is,
\[
 \cH = \cU+\cV \quad\text{ and }\quad \cU\cap\cV=\{0\}\,,
\]
is said to \emph{decompose} the operator $B$ if both $\cU$ and $\cV$ are invariant for $B$ and the domain $\Dom(B)$ splits as
\begin{equation}\label{eq:domSplit}
 \Dom(B) = \bigl(\Dom(B)\cap\cU\bigr) + \bigr(\Dom(B)\cap\cV\bigr)\,.
\end{equation}

The following well-known result provides a characterization of decomposing pairs of subspaces in the case where the operator
$B-\lambda$ is bijective for some constant $\lambda$.

\begin{lemma}[cf.\ {\cite[Remark 2.3 and Lemma 2.4]{TW14}}]\label{lem:resInv}
 Suppose that for some constant $\lambda$ the operator $B-\lambda$ is bijective. Then, a pair of complementary subspaces $\cU$ and
 $\cV$ decomposes the operator $B$ if and only if both subspaces are invariant for $B$ and $(B-\lambda)^{-1}$.

 \begin{proof}
  First, suppose that $\cU$ and $\cV$ decompose $B$. We have to show that $\cU$ and $\cV$ are invariant for $(B-\lambda)^{-1}$. Let
  $x\in\cU$  be arbitrary. Since one has $(B-\lambda)^{-1}x\in\Dom(B)$, the splitting \eqref{eq:domSplit} yields
  \[
   (B-\lambda)^{-1}x = u + v
  \]
  for some $u\in\Dom(B)\cap\cU$ and $v\in\Dom(B)\cap\cV$. Thus,
  \[
   \cU\ni x = (B-\lambda)(u+v) = (B-\lambda)u + (B-\lambda)v\,.
  \]
  Taking into account that $\cU$ and $\cV$ are invariant for $B-\lambda$ and that $\cU$ and $\cV$ are complementary, one concludes
  that $(B-\lambda)v=0$, that is, $v=0$. Hence,
  \[
   (B-\lambda)^{-1}x = u\in\cU\,,
  \]
  so that $\cU$ is invariant for $(B-\lambda)^{-1}$. Analogously, $\cV$ is invariant for $(B-\lambda)^{-1}$.

  Conversely, suppose that $\cU$ and $\cV$ are both invariant for $B$ and $(B-\lambda)^{-1}$. Let $x\in\Dom(B)$ be arbitrary. Then,
  one has
  \[
   (B-\lambda)x = u + v
  \]
  for some $u\in\cU$ and $v\in\cV$. Since $\cU$ and $\cV$ are invariant for $(B-\lambda)^{-1}$, one concludes that
  \[
   x = (B-\lambda)^{-1}u + (B-\lambda)^{-1}v \in \bigl(\Dom(B)\cap\cU\bigr) + \bigl(\Dom(B)\cap\cV\bigr)\,.
  \]
  Hence, $\Dom(B)$ splits as in \eqref{eq:domSplit}.
 \end{proof}%
\end{lemma}

As a consequence, if a pair of subspaces $\cU$ and $\cV$ decomposes an operator $B$, then it follows from Lemma \ref{lem:resInv}
that for every constant $\lambda$ the operator $B-\lambda$ is bijective if and only if its parts, the restrictions $B|_\cU-\lambda$
and $B|_\cV-\lambda$, are both bijective. In particular, if $B$ is additionally assumed to be closed, then both parts $B|_\cU$ and
$B|_\cV$ are closed and we have the spectral identity
\[
 \spec(B) = \spec(B|_\cU) \cup \spec(B|_\cV)\,.
\]

For the rest of this paper we make the following assumptions that introduce a general framework for the off-diagonal perturbation
theory.
\begin{hypothesis}\label{hyp:blockMatrix}
 Given an orthogonal decomposition $\cH=\cH_0\oplus\cH_1$ of the Hilbert space, let
 \[
  \begin{pmatrix} A_0 & W_1\\ W_0 & A_1 \end{pmatrix} =
  \begin{pmatrix} A_0 & 0\\ 0 & A_1 \end{pmatrix} + \begin{pmatrix} 0 & W_1\\ W_0 & 0 \end{pmatrix} =: A + V\,.
 \]
 be a possibly unbounded $2\times2$ block operator matrix on the natural domain
 \[
  \dav := \Dom(A+V) = \bigl(\Dom(A_0)\cap\Dom(W_0)\bigr) \oplus \bigl(\Dom(A_1)\cap\Dom(W_1)\bigr) =: \dav_0 \oplus \dav_1\,.
 \]
\end{hypothesis}

Recall that a closed subspace $\cG\subset\cH$ is said to be a \emph{graph subspace} associated with a closed subspace
$\cN\subset\cH$ and a bounded operator $X$ from $\cN$ to its orthogonal complement $\cN^\perp$ if
\[
 \cG=\cG(\cN, X):=\{x\in \cH\,|\, P_{\cN^\perp }x=XP_\cN x\}\,,
\]
where $P_\cN$ and $P_{\cN^\perp}$ denote the orthogonal projections onto $\cN$ and $\cN^\perp$, respectively. Here, the operator
$X$ has been identified with its trivial continuation to the whole Hilbert space $\cH$ and this identification is used throughout
the paper.

In this context, the equivalence between the assertions i) and ii) of Theorem \ref{thm:main} is just a particular case of Lemma
\ref{lem:resInv}:
\begin{proof}[Proof of i) $\Leftrightarrow$ ii) in Theorem \ref{thm:main}]
 Apply Lemma \ref{lem:resInv} to the pair of graphs $\cG_0=\cG(\cH_0,X_0)$ and $\cG_1=\cG(\cH_1,X_1)$.
\end{proof}%

For the rest of the proof of Theorem \ref{thm:main}, we need the following well-known invariance criterion for graph subspaces,
see, e.g., \cite{Tre08} and references therein.
\begin{lemma}\label{lem:invariance}
 Assume Hypothesis \ref{hyp:blockMatrix}. The graph $\cG(\cH_0,X_0)$ of some bounded linear operator $X_0\colon\cH_0\to\cH_1$ is
 invariant for $A+V$ if and only if $X_0$ satisfies the Riccati equation
 \begin{equation}\label{eq:RiccatiX0}
  A_1X_0f - X_0A_0f - X_0W_1X_0f + W_0f = 0
 \end{equation}
 for all $f\in\dav_0\subset\cH_0$ such that $X_0f\in\dav_1$.

 Analogously, the graph $\cG(\cH_1,X_1)$ of a bounded linear operator $X_1\colon\cH_1\to\cH_0$ is invariant for $A+V$ if and only
 if
 \begin{equation}\label{eq:RiccatiX1}
  A_0X_1g - X_1A_1g - X_1W_0X_1g + W_1g = 0
 \end{equation}
 for all $g\in\dav_1\subset\cH_1$ such that $X_1g\in\dav_0$.
 
 \begin{proof}
  Define
  \[
   \cD_0 := \{ f\in\dav_0 \mid X_0f \in\dav_1\}\,.
  \]
  Observing that
  \[
   (A+V)\begin{pmatrix} f\\ X_0f \end{pmatrix} = \begin{pmatrix} A_0f + W_1X_0f\\ W_0f + A_1X_0f \end{pmatrix}
   \quad\text{ for }\quad f\in\cD_0\,,
  \]
  and taking into account that $\dav\cap\cG(\cH_0,X_0)=\{f\oplus X_0f \mid f\in\cD_0\}$, one concludes that the graph
  $\cG(\cH_0,X_0)$ is invariant for $A+V$ if and only if $(W_0 + A_1X_0)f=X_0(A_0 + W_1X_0)f$ holds for all $f\in\cD_0$. This, in
  turn, can be rewritten as \eqref{eq:RiccatiX0}. The second part is proved analogously by changing the roles of $\cH_0$ and
  $\cH_1$.
 \end{proof}%
\end{lemma}

\begin{remark}\label{rem:strong}
 If, in the situation of Lemma \ref{lem:invariance}, it is known in advance that
 \begin{equation}\label{eq:incl0}
  \Ran (X_0|_{\dav_0})\subset \dav_1
 \end{equation}
 and 
 \begin{equation}\label{eq:incl1}
  \Ran (X_1|_{\dav_1})\subset \dav_0\,,
 \end{equation}
 then the Riccati equations \eqref{eq:RiccatiX0} and \eqref{eq:RiccatiX1} hold for all $f\in\dav_0$ and $g\in\dav_1$, respectively.
 In this case, the operators $X_0$ and $X_1$ are called \emph{strong solutions} to the corresponding operator Riccati equations
 \[
  A_1X_0 - X_0A_0 - X_0W_1X_0 + W_0= 0
 \]
 and
 \[
  A_0X_1 - X_1A_1 - X_1W_0X_1 + W_1 = 0\,, 
 \]
 see \cite[Section 3]{AMM03}.
 
 It is worth noting that with the above mentioned identification for the operators $X_0$ and $X_1$, the inclusions \eqref{eq:incl0}
 and \eqref{eq:incl1} simply mean that $\Dom(A+V)$ is invariant for $X_0$ and $X_1$, respectively, cf.\ assertion iii) of Theorem
 \ref{thm:main}.
\end{remark}

\section{The first Proof of Theorem \ref{thm:main}. Block diagonalizations}\label{sec:mainThm}

In this section, the equivalence between the assertions i) and iii) of Theorem \ref{thm:main} is established and, at the same time,
the block diagonalization \eqref{mainfac} is derived. The latter is compared to previously known results in the literature, see
Remark \ref{rem:alternative} below.

The initial point of our considerations are the above mentioned Riccati equations \eqref{eq:RiccatiX0} and \eqref{eq:RiccatiX1} in
Lemma \ref{lem:invariance}, so we shall start with a closer inspection of them:

In the situation of Hypothesis \ref{hyp:blockMatrix}, suppose that two complementary graphs $\cG(\cH_0,X_0)$ and $\cG(\cH_1,X_1)$
associated with bounded linear operators $X_0\colon\cH_0\to\cH_1$ and $X_1\colon\cH_1\to\cH_0$, respectively, are invariant for
$A+V$. Then, the two Riccati equations \eqref{eq:RiccatiX0} and \eqref{eq:RiccatiX1} hold simultaneously and can therefore be
combined to a single block Riccati equation for the operator
\begin{equation}\label{eq:defY}
 Y = \begin{pmatrix} 0 & X_1\\ X_0 & 0 \end{pmatrix}.
\end{equation}
Namely,
\begin{equation}\label{eq:Ricc}
 AYx - YAx - YVYx + Vx = 0 \quad\text{ for }\quad x\in\cD\,,
\end{equation}
where
\begin{equation}\label{eq:defD}
 \cD := \{ x\in\dav \mid Yx\in\dav\}\,.
\end{equation}
In turn, this block Riccati equation can be rewritten as
\begin{equation}\label{eq:blockDiagStrong}
 (I_\cH-Y)(A+V)x = (A-YV)(I_\cH-Y)x \quad\text{ for }\quad x\in\cD\,.
\end{equation}

Here, as the following lemma shows, the operator $I_\cH-Y$ is an automorphism of $\cH$.

\begin{lemma}\label{lem:automorphisms}
 Let $\cG(\cH_0,X_0)$ and $\cG(\cH_1,X_1)$ be complementary graph subspaces. Then, the operators $I_\cH\pm Y$ with $Y$ defined as
 in \eqref{eq:defY} are automorphisms of $\cH$.

 \begin{proof} 
  First, observe that
  \begin{equation}\label{eq:Jsimilar}
   J(I_\cH-Y)J = I_\cH-JYJ= I_\cH + Y\,,
  \end{equation}
  where $J$ is the unitary block diagonal matrix given by 
  \begin{equation}\label{eq:defJ}
   J = \begin{pmatrix} I_{\cH_0} & 0\\ 0 & -I_{\cH_1} \end{pmatrix}.
  \end{equation}
 
  Since $I_\cH+Y$ maps $\cH_0$ and $\cH_1$ bijectively onto $\cG(\cH_0,X_0)$ and $\cG(\cH_1,X_1)$, respectively, and since the
  graphs are complementary subspaces, one concludes that $I_\cH+Y$ is an automorphism, and so is $I_\cH-Y$ by \eqref{eq:Jsimilar}.
 \end{proof}%
\end{lemma}

\begin{remark}\label{rem:automorphism}
 \begin{enumerate}
  \item The notion `automorphism' in Lemma \ref{lem:automorphisms} can be understood either algebraically or topologically since by
        the closed graph theorem every bounded bijective operator on $\cH$ automatically has a bounded inverse.
  \item It is easy to see from the proof of Lemma \ref{lem:automorphisms} that also the converse statement of Lemma
        \ref{lem:automorphisms} is valid, that is, the graphs $\cG(\cH_0,X_0)$ and $\cG(\cH_1,X_1)$ are complementary if $I_\cH+Y$
        (resp.\ $I_\cH-Y$) is an automorphism of $\cH$. It is also worth noting that the latter will automatically be the case if
        $\norm{Y}<1$.
 \end{enumerate}
\end{remark}

The following proposition, the proof of which relies upon \eqref{eq:blockDiagStrong}, is the core of our proof of the equivalence
between the assertions i) and iii) of Theorem \ref{thm:main}.

\begin{proposition}\label{prop:extRel}
 Assume Hypothesis \ref{hyp:blockMatrix}. Let $\cG(\cH_0,X_0)$ and $\cG(\cH_1,X_1)$ with bounded linear operators
 $X_0\colon\cH_0\to\cH_1$ and $X_1\colon\cH_1\to\cH_0$, respectively, be two complementary graph subspaces, and suppose that both
 are invariant for $A+V$. Define $Y$ as in \eqref{eq:defY}. Then:
 \begin{enumerate}
  \item The domain $\dav=\Dom(A+V)$ splits as
        \begin{equation}\label{eq:domSplit2}
         \dav = \bigl(\dav\cap\cG(\cH_0,X_0)\bigr) + \bigl(\dav\cap\cG(\cH_1,X_1)\bigr)
        \end{equation}
        if and only if one has the extension relation
        \begin{equation}\label{eq:extRelSup}
         (I_\cH-Y)(A+V) \supset (A-YV)(I_\cH-Y)\,.
        \end{equation}
  \item The domain $\Dom(A+V)$ is invariant for $X_0$ and $X_1$ if and only if one has the extension relation
        \begin{equation}\label{eq:extRelSub}
         (I_\cH-Y)(A+V) \subset (A-YV)(I_\cH-Y)\,.
        \end{equation}
 \end{enumerate}

 \begin{proof}
  Since the graphs $\cG(\cH_0,X_0)$ and $\cG(\cH_1,X_1)$ are invariant for $A+V$ by hypothesis, the Riccati equation in the form
  \eqref{eq:blockDiagStrong} holds with the set $\cD$ given as in \eqref{eq:defD}. Moreover, since the operator matrix $Y$ is
  off-diagonal, this set $\cD$, similar to the domain $\dav=\Dom(A+V)$, splits as
  \[
   \cD = \bigl(\cD\cap\cH_0\bigr) + \bigl(\cD\cap\cH_1\bigr) =: \cD_0 + \cD_1\,.
  \]
  In particular, the block diagonal operator $J$ in \eqref{eq:defJ} maps both sets $\dav$ and $\cD$ onto themselves.
 
  (a). It is easy to see that $\dav\cap\cG(\cH_0,X_0)=\{f\oplus X_0f \mid f\in\cD_0\}$ and, similarly, that
  $\dav\cap\cG(\cH_1,X_1)=\{X_1g\oplus g \mid g\in\cD_1\}$. Hence,
  \begin{equation}\label{eq:graphSumRepr}
   \bigl(\dav\cap\cG(\cH_0,X_0)\bigr) + \bigl(\dav\cap\cG(\cH_1,X_1)\bigr) = \Ran\bigl((I_\cH+Y)|_\cD\bigr)\,.
  \end{equation}

  Suppose that $\dav$ splits as in \eqref{eq:domSplit2}. Equation \eqref{eq:graphSumRepr} then implies that
  $\dav=\Ran\bigl((I_\cH+Y)|_\cD\bigr)$. In view of the similarity \eqref{eq:Jsimilar} and the fact that $J$ maps $\cD$ and $\dav$
  onto themselves, this yields
  \[
   \dav = \Ran\bigl((I_\cH-Y)|_\cD\bigr)\,.
  \]
  Thus, the operator on the right-hand side of \eqref{eq:blockDiagStrong} has natural domain $\cD\subset\dav=\Dom(A+V)$, and,
  therefore, equation \eqref{eq:blockDiagStrong} agrees with the extension relation \eqref{eq:extRelSup}.
   
  Conversely, suppose that \eqref{eq:extRelSup} holds. Let $x\in\dav$ be arbitrary, and write $x=(I_\cH+Y)y$ with $y\in\cH$. In
  view of \eqref{eq:Jsimilar}, one has
  \[
   (I_\cH-Y)Jy = Jx\in\dav\,,
  \]
  and, therefore,
  \[
   Jy \in\Dom\bigl((A-YV)(I_\cH-Y)\bigr) \subset \Dom(A+V) = \dav\,.
  \]
  In turn, this yields $y\in\dav$. Rewriting the definition of $y$ as $Yy=x-y\in\dav$, one concludes that $y\in\cD$. Thus,
  $\dav\subset\Ran\bigl((I_\cH+Y)|_\cD\bigr)$, and the identity \eqref{eq:graphSumRepr} implies that $\dav$ splits as in
  \eqref{eq:domSplit2}.
 
  (b) First, observe that $\dav=\Dom(A+V)$ is invariant for $X_0$ and $X_1$ if and only if it is invariant for $Y$, that is,
  $\cD=\dav$. It therefore suffices to show that the extension relation \eqref{eq:extRelSub} holds if and only if $\cD=\dav$.
 
  Suppose that $\cD=\dav$. Then, equation \eqref{eq:blockDiagStrong} holds for all $x\in\dav=\Dom(A+V)$, which agrees with the
  extension relation \eqref{eq:extRelSub}.
 
  Conversely, suppose that \eqref{eq:extRelSub} holds. Since the operator on the left-hand side of \eqref{eq:extRelSub} has domain
  $\Dom(A+V)=\dav$, the operator $I_\cH-Y$ then maps $\dav$ into $\Dom(A-YV)=\dav$. Hence, $Y$ maps $\dav$ into itself, which means
  that $\cD=\dav$.
 \end{proof}%
\end{proposition}

In view of Proposition \ref{prop:extRel}, the proof of the equivalence i) $\Leftrightarrow$ iii) in Theorem \ref{thm:main} reduces
to the following statement: Under the assumptions of Theorem \ref{thm:main}, each of the extension relations \eqref{eq:extRelSup}
and \eqref{eq:extRelSub} implies the other and therefore the operator equality
\begin{equation}\label{eq:blockSimilar}
 (I_\cH-Y)(A+V) = (A-YV)(I_\cH-Y)\,
\end{equation}
holds.

In order to show the latter statement, we need the following elementary observation.

\begin{lemma}[{\cite[Lemma 1.3]{Schm12}}]\label{lem:schm}
 Let $\cT$ and $\cS$ be linear operators such that $\cS\subset\cT$. If $\cS$ is surjective and $\cT$ is injective, then $\cS=\cT$.
 
 \begin{proof}
  For the sake of completeness, we reproduce the proof from \cite{Schm12}.

  Let $y\in\Dom(\cT)$ be arbitrary. Since $\cS$ is surjective, there is $x\in\Dom(\cS) \subset\Dom(\cT)$ such that
  $\cT y=\cS x=\cT x$, where we have taken into account that $\cS\subset\cT$. The injectivity of $\cT$ now implies that
  $y=x\in\Dom(\cS)$. Thus, $\Dom(\cT)=\Dom(\cS)$ and, hence, $\cS=\cT$.
 \end{proof}%
\end{lemma}

We are now ready to prove Theorem \ref{thm:main}.

\begin{proof}[Proof of i) $\Leftrightarrow$ iii) in Theorem \ref{thm:main}]
 Since the two graphs $\cG(\cH_0,X_0)$ and $\cG(\cH_1,X_1)$ are complementary subspaces, the operator $I_\cH-Y$ is an automorphism
 of $\cH$ by Lemma \ref{lem:automorphisms}. Hence, the operators $(I_\cH-Y)(A+V-\lambda)$ and $(A-YV-\lambda)(I_\cH-Y)$ are both
 bijective. 
 
 The claim then is an immediate consequence of Proposition \ref{prop:extRel} and Lemma \ref{lem:schm} since every extension
 relation between the operators $(I_\cH-Y)(A+V)$ and $(A-YV)(I_\cH-Y)$ directly translates to the corresponding one between
 $(I_\cH-Y)(A+V-\lambda)$ and $(A-YV-\lambda)(I_\cH-Y)$ and vice versa.
\end{proof}%

\begin{remark}\label{rem:closed}
 Although Theorem \ref{thm:main} has been stated for closed operators $A+V$ and $A-YV$, the above proof shows that $A+V-\lambda$
 and $A-YV-\lambda$ only need to be bijective. In fact, a closer inspection of the proof shows that for each implication only the
 injectivity of one of the operators and the surjectivity of the other is required. Namely, in the proof of the implication i)
 $\Rightarrow$ iii), we only use that $A+V-\lambda$ is injective and $A-YV-\lambda$ is surjective. For the converse, only
 surjectivity of $A+V-\lambda$  and injectivity of $A-YV-\lambda$ is used, cf.\ also Remark \ref{rem:injsursingle} below.
\end{remark}

As a byproduct of our considerations, the operator equality \eqref{eq:blockSimilar} can be rewritten as
\begin{equation}\label{eq:blockDiag2ndMain}
 (I_\cH-Y) (A+V) (I_\cH-Y)^{-1} = A-YV = \begin{pmatrix} A_0-X_1W_0 & 0\\ 0 & A_1-X_0W_1 \end{pmatrix},
\end{equation}
so that the spectral identity 
\begin{equation}\label{eq:specid1}
\spec (A+V)=\spec( A_0-X_1W_0) \cup  \spec (A_1-X_0W_1)
\end{equation}
holds. In fact, the restrictions of $A+V$ to the graphs $\cG(\cH_0,X_0)$ and $\cG(\cH_1,X_1)$ are similar to the diagonal entries
$A_0-X_1W_0$ and $A_1-X_0W_1$, respectively. In order to see this, observe that
\[
 (I_\cH-Y)(I_\cH+Y) = I_\cH-Y^2 = \begin{pmatrix} I_{\cH_0}-X_1X_0 & 0\\ 0 & I_{\cH_1}-X_0X_1 \end{pmatrix}.
\]
Taking into account Lemma \ref{lem:automorphisms}, equation \eqref{eq:blockDiag2ndMain} then turns into
\begin{equation}\label{eq:blockDiagExt}
 (I_\cH+Y)^{-1} (A+V) (I_\cH+Y) = (I_\cH-Y^2)^{-1} (A-YV) (I_\cH-Y^2)\,,
\end{equation}
where the right-hand side is again block diagonal with entries similar to the diagonal entries of $A-YV$. Since $I_\cH+Y$ maps
$\cH_0$ and $\cH_1$ bijectively onto the graphs $\cG(\cH_0,X_0)$ and $\cG(\cH_1,X_1)$, respectively, this proves the claim for the
parts of $A+V$.

It is a direct consequence of Proposition \ref{prop:extRel} that the block diagonalization \eqref{eq:blockDiag2ndMain}
(resp.\ \eqref{eq:blockDiagExt}) holds as soon as the pair of graphs $\cG(\cH_0,X_0)$ and $\cG(\cH_1,X_1)$ decomposes the operator
$A+V$ and $\Dom(A+V)$ is invariant for both $X_0$ and $X_1$. In turn, the operators $A+V-\lambda$ and $A-YV-\lambda$ then are
bijective for the same constants $\lambda$. In this respect, if $A+V$ is closed with non-empty resolvent set, the hypothesis of
intersecting resolvent sets in Theorem \ref{thm:main} is not only sufficient but also necessary for the equivalence
i) $\Leftrightarrow$ iii) to hold.

\begin{remark}\label{rem:alternative}
 Similar to \eqref{eq:blockDiagStrong}, equation \eqref{eq:Ricc} can alternatively be rewritten as
 \begin{equation}\label{eq:qsimilar(A+V)T}
  (A+V)(I_\cH+Y)x = (I_\cH+Y)(A+VY)x\quad\text{ for }\quad x\in\cD\,,
 \end{equation}
 with the block diagonal operator
 \[
  A+VY = \begin{pmatrix} A_0+W_1X_0 & 0\\ 0 & A_1+W_0X_1 \end{pmatrix}.
 \]
 In particular, if \eqref{eq:blockDiagExt} holds, one has
 \begin{equation}\label{eq:diagSimilar}
  (I_\cH-Y^2)^{-1} (A-YV) (I_\cH-Y^2)x = (A+VY)x \quad\text{for}\quad x\in\cD\,.
 \end{equation}
 Block diagonalizations based on \eqref{eq:qsimilar(A+V)T} have already been considered in the literature, see, e.g.,
 \cite{ADL01,AMM03,Cue12,LT97,Tre08}, whereas \eqref{eq:blockDiag2ndMain} and \eqref{eq:blockDiagExt}, to the best of our
 knowledge, appear only in the present work and to some extent in the authors (unpublished) preprint \cite{MSS13} and seem to be new.

 The obvious difference between \eqref{eq:blockSimilar} and \eqref{eq:qsimilar(A+V)T} is the appearance of the block diagonal
 operators $A-YV$ and $A+VY$, respectively, where in the former case the operator $Y$ stands to the left of $V$, in the latter one
 it stands to the right of $V$. In particular, the operator $A-YV$ always has stable domain $\Dom(A-YV)=\Dom(A+V)$, whereas the
 natural domain of $A+VY$ may depend on $Y$. More precisely, the domain of $A+VY$ only satisfies
 \[
  \cD \subset \Dom(A+VY) \subset \Dom(A)\,,
 \]
 and either one of the inclusions may a priori be strict. Moreover, neither of the inclusions $\Dom(A+VY)\subset\Dom(A+V)$ and
 $\Dom(A+V)\subset\Dom(A+VY)$ is self-evident. In view of \eqref{eq:diagSimilar}, the block diagonalization \eqref{eq:blockDiagExt}
 therefore seems to extend \eqref{eq:qsimilar(A+V)T}.

 However, in the particular case where the off-diagonal part $V$ satisfies the additional requirement that $\Dom(A)\subset\Dom(V)$,
 equation \eqref{eq:blockDiag2ndMain} implies that
 \[
  \cD=\Dom(A+V)=\Dom(A)=\Dom(A+VY)\,.
 \]
 It is then easy to see from \eqref{eq:diagSimilar} that the right-hand side of \eqref{eq:blockDiagExt} actually agrees with
 $A+VY$, so that
 \begin{equation}\label{eq:blockDiag1stMain}
  (I_\cH+Y)^{-1}(A+V)(I_\cH+Y) = A+VY=\begin{pmatrix} A_0+W_1X_0 & 0\\ 0 & A_1+W_0X_1 \end{pmatrix}.
 \end{equation} 
 In particular, the restrictions of $A+V$ to the graphs $\cG(\cH_0,X_0)$ and $\cG(\cH_1,X_1)$ are similar to the diagonal blocks
 $A_0+W_1X_0$ and $A_1+W_0X_1$, respectively, and one has
 \begin{equation}\label{eq:specid2}
  \spec (A+V)=\spec(A_0+W_1X_0 ) \cup \spec ( A_1+W_0X_1)\,.
 \end{equation}
 Some further discussions on this matter can be found in Section 2.6 of the authors preprint \cite{MSS13}.

 We revisit the case $\Dom(A)\subset\Dom(V)$ in Sections \ref{sec:app} and \ref{sec:Example} below.
\end{remark}

\section{The second Proof of Theorem \ref{thm:main}. The single graph subspace approach}\label{sec:general}

The property of a pair of graph subspaces to decompose an operator clearly involves both subspaces simultaneously, whereas the
invariance of the domain for the corresponding angular operators is definitely a property of the separate graphs. At this point,
Lemma \ref{lem:resInv} provides a natural substitute for the splitting property of the operator domain by the invariance of the
separate graphs for the resolvent of the operator. This allows one to deal with each graph separately.

In this section, we therefore concentrate on single graph subspaces. For definiteness, we state the results for graphs with respect
to $\cH_0$. The corresponding results for graphs with respect to $\cH_1$ can be obtained just by switching the roles of $\cH_0$ and
$\cH_1$. In the alternative proof of Theorem \ref{thm:main}, the results are then applied to each of the two graphs
$\cG(\cH_0,X_0)$ and $\cG(\cH_1,X_1)$.

We begin with the following lemma, which is nothing but a corollary to Lemma \ref{lem:invariance}.

\begin{lemma}\label{lem:triDiag}
 Assume Hypothesis \ref{hyp:blockMatrix}, and let $X_0\colon\cH_0\to\cH_1$ be a bounded linear operator such that its graph
 $\cG(\cH_0,X_0)$ is invariant for $A+V$. If, in addition, $\dav=\Dom(A+V)$ is invariant for $X_0$, then one has the operator
 identity
 \begin{equation}\label{eq:triDiag}
  \begin{pmatrix} I_{\cH_0}&0\\ -X_0&I_{\cH_1} \end{pmatrix} (A+V) \begin{pmatrix} I_{\cH_0}&0\\ X_0&I_{\cH_1}\end{pmatrix} =
  \begin{pmatrix} (A_0+W_1X_0)|_{\dav_0} & W_1\\ 0 & A_1-X_0W_1 \end{pmatrix}.
 \end{equation}
 In particular, the restriction of $A+V$ to $\cG(\cH_0,X_0)$ is similar to $(A_0+W_1X_0)|_{\dav_0}$, namely
 \begin{equation}\label{eq:similar}
  T^{-1}(A+V)T = (A_0+W_1X_0)|_{\dav_0}\,,
 \end{equation}
 where $T\colon\cH_0\to\cG(\cH_0,X_0)$ is given by
 \[
  Tx := x\oplus X_0x\,.
 \]

 \begin{proof}
  If $\dav$ is invariant for $X_0$, that is, if $X_0$ maps $\mathfrak{D}_0$ into $\dav_1$, then the operator on the right-hand side
  of \eqref{eq:triDiag} is well-defined with natural domain $\Dom(A+V)=\dav_0\oplus\dav_1$ and the bijective operator
  $\begin{pmatrix} I_{\cH_0} & 0\\ X_0 & I_{\cH_1} \end{pmatrix}$ maps $\Dom(A+V)$ onto itself. The rest follows by direct
  computation using the Riccati equation \eqref{eq:RiccatiX0} from Lemma \ref{lem:invariance}; cf.\ also Remark \ref{rem:strong}.
 \end{proof}%
\end{lemma}

It is interesting to note that in the block diagonal representations \eqref{eq:blockDiag1stMain} and \eqref{eq:blockDiag2ndMain},
the $X$- and $W$-operators are ordered `alphabetically' and `reversed alphabetically', respectively, while in the block triangular
representation \eqref{eq:triDiag} the order of the corresponding operators in the diagonal entries is mixed.

The following theorem, the proof of which relies upon the identity \eqref{eq:triDiag} in Lemma \ref{lem:triDiag}, represents the
core of our considerations for single graphs.

\begin{theorem}\label{thm:single}
 Assume Hypothesis \ref{hyp:blockMatrix}, and let $X_0\colon\cH_0\to\cH_1$ be a bounded linear operator such that its graph
 $\cG(\cH_0,X_0)$ is invariant for $A+V$. Suppose, in addition, that for some constant $\lambda$ the operators $A+V-\lambda$ and
 $A_1-\lambda-X_0W_1$ are bijective.

 Then, the graph $\cG(\cH_0,X_0)$ is invariant for $(A+V-\lambda)^{-1}$ if and only if $\Dom(A+V)$ is invariant for $X_0$.

 \begin{proof}
  Suppose that $\cG(\cH_0,X_0)$ is invariant for $(A+V-\lambda)^{-1}$. We need to show that $X_0$ maps $\dav_0$ into $\dav_1$. To
  this end, let $f\in\dav_0$ be arbitrary. Since $A_1-\lambda-X_0W_1$ is surjective, there is $g\in\dav_1=\Dom(A_1-X_0W_1)$ such
  that
  \[
   (A_1-\lambda-X_0W_1)g = X_0(A_0-\lambda)f - W_0f\,.
  \]
  This can be rewritten as
  \[
   (A_1-\lambda)g + W_0f = X_0\bigl((A_0-\lambda)f + W_1g\bigr)\,,
  \]
  so that
  \[
   (A+V-\lambda)\begin{pmatrix} f\\ g \end{pmatrix} = \begin{pmatrix} (A_0-\lambda)f + W_1g\\ (A_1-\lambda)g + W_0f \end{pmatrix}
   = \begin{pmatrix} (A_0-\lambda)f + W_1g\\ X_0\bigl((A_0-\lambda)f + W_1g\bigr) \end{pmatrix} \in\cG(\cH_0,X_0)\,.
  \]
  Since $\cG(\cH_0,X_0)$ is invariant for $(A+V-\lambda)^{-1}$, one concludes that $f\oplus g\in\cG(\cH_0,X_0)$, that is,
  $X_0f=g\in\dav_1$. This shows that $X_0$ maps $\dav_0$ into $\dav_1$, that is, $\Dom(A+V)$ is invariant for the operator $X_0$.

  Conversely, suppose that $\Dom(A+V)$ is invariant for $X_0$. Let $x=x_0\oplus X_0x_0\in\cG(\cH_0,X_0)$ be arbitrary. Since
  $A+V-\lambda$ is surjective, there is $f\oplus g\in\dav_0\oplus\dav_1$ such that
  \[
   \begin{pmatrix} x_0\\ X_0x_0 \end{pmatrix} = (A+V-\lambda) \begin{pmatrix} f\\ g\end{pmatrix}.
  \]
  We have to show that $f\oplus g\in\cG(\cH_0,X_0)$, that is, $g=X_0f$.

  Clearly,
  \[
   \begin{aligned}
    \begin{pmatrix} x_0\\ 0 \end{pmatrix}
    &= \begin{pmatrix} I_{\cH_0} & 0\\ -X_0 & I_{\cH_1} \end{pmatrix} \begin{pmatrix} x_0\\ X_0x_0 \end{pmatrix}
     = \begin{pmatrix} I_{\cH_0} & 0\\ -X_0 & I_{\cH_1} \end{pmatrix}(A+V-\lambda) \begin{pmatrix} f\\ g\end{pmatrix}\\
    &= \begin{pmatrix} I_{\cH_0} & 0\\ -X_0 & I_{\cH_1} \end{pmatrix}(A+V-\lambda)
       \begin{pmatrix} I_{\cH_0} & 0\\ X_0 & I_{\cH_1} \end{pmatrix} \begin{pmatrix} f\\ g-X_0f \end{pmatrix}.
   \end{aligned}
  \]
  Since $\Dom(A+V)=\dav_0\oplus\dav_1$ is invariant for $X_0$, one has $g-X_0f\in\dav_1$. The identity \eqref{eq:triDiag} in Lemma
  \ref{lem:triDiag} therefore yields
  \[
   \begin{aligned}
    \begin{pmatrix} x_0\\ 0\end{pmatrix}
    &= \begin{pmatrix} (A_0-\lambda+W_1X_0)|_{\mathfrak{D}_0} & W_1\\ 0 & A_1-\lambda-X_0W_1 \end{pmatrix}
       \begin{pmatrix} f\\ g-X_0f \end{pmatrix}\\
    &= \begin{pmatrix} (A_0-\lambda+W_1X_0)f + W_1(g-X_0f)\\ (A_1-\lambda-X_0W_1)(g-X_0f) \end{pmatrix}.
   \end{aligned}
  \]
  Hence, $(A_1-\lambda-X_0W_1)(g-X_0f)=0$, and the injectivity of $A_1-\lambda-X_0W_1$ yields $g=X_0f$. Thus,
  $(A+V-\lambda)^{-1}x=f\oplus X_0f\in\cG(\cH_0,X_0)$, which shows that $\cG(\cH_0,X_0)$ is invariant for $(A+V-\lambda)^{-1}$.
 \end{proof}%
\end{theorem}

\begin{remark}\label{rem:injsursingle}
 In the situation of Theorem \ref{thm:single}, the invariance of $\cG(\cH_0,X_0)$ for $(A+V-\lambda)^{-1}$ is clearly equivalent
 to the identity
 \begin{equation}\label{eq:surjective}
  \Ran\bigl((A+V-\lambda)|_{\Dom(A+V)\cap\cG(\cH_0,X_0)}\bigr) = \cG(\cH_0,X_0)\,.
 \end{equation}

 Upon this observation, the proof of Theorem \ref{thm:single} actually shows the following more general statements (cf. Remark
 \ref{rem:closed} for the corresponding observation in the block approach) where the operators $A+V-\lambda$ and
 $A_1-\lambda-X_0W_1$ are not assumed to be bijective:
 \begin{enumerate}
  \renewcommand{\theenumi}{\alph{enumi}}
  \item Suppose that \eqref{eq:surjective} holds. Then, $\Dom(A+V)$ is invariant for $X_0$ if $A+V-\lambda$ is injective and
        $A_1-\lambda-X_0W_1$ is surjective.
  \item Let $\Dom(A+V)$ be invariant for $X_0$. Then \eqref{eq:surjective} holds if $A+V-\lambda$ is surjective and
        $A_1-\lambda-X_0W_1$ is injective.
 \end{enumerate}

 However, if $A+V-\lambda$ is bijective, then the bijectivity of $A_1-\lambda-X_0W_1$ is a natural requirement in Theorem
 \ref{thm:single}. Indeed, if \eqref{eq:surjective} holds and $\Dom(A+V)$ is invariant for $X_0$, then it follows from the
 identities \eqref{eq:triDiag} and \eqref{eq:similar} in Lemma \ref{lem:triDiag} that along with $A+V-\lambda$ the operators
 $(A_0-\lambda+W_1X_0)|_{\dav_0}$ and $A_1-\lambda-XW_1$ are bijective. Hence, in both cases (a) and (b) the operator
 $A_1-\lambda-XW_1$ turns out to be bijective in the end as well.
\end{remark}

We now turn to the second proof of Theorem \ref{thm:main}.

\begin{proof}[Direct proof of ii) $\Leftrightarrow$ iii) in Theorem \ref{thm:main}] 
 First, observe that the bijectivity of the block diagonal operator
 \[
  A-\lambda-YV = \begin{pmatrix} A_0-\lambda-X_1W_0 & 0\\ 0 & A_1-\lambda-X_0W_1 \end{pmatrix}
 \]
 means bijectivity of both $A_0-\lambda-X_1W_0$ and $A_1-\lambda-X_0W_1$. Hence, Theorem \ref{thm:single} can be applied to each
 graph $\cG(\cH_0,X_0)$ and (with roles of $\cH_0$ and $\cH_1$ switched) $\cG(\cH_1,X_1)$. This yields that the graphs
 $\cG(\cH_0,X_0)$ and $\cG(\cH_1,X_1)$ are invariant for $(A+V-\lambda)^{-1}$ if and only if $\Dom(B)=\Dom(A+V)$ is invariant for
 both $X_0$ and $X_1$.
\end{proof}%

The statement that invariance of $\cG(\cH_0,X_0)$ for $A+V$ implies invariance of $\Dom(A+V)$ for $X_0$ has already been discussed
under certain assumptions in \cite[Theorem 4.1]{LT97}, \cite[Lemma 6.1]{TW14}, and \cite[Proposition 7.5]{W11}. Those results
appear to be particular cases of Theorem \ref{thm:single} above, cf.\ also  Section \ref{sec:app} below. Moreover, Theorem
\ref{thm:single} also provides the converse statement, which, to the best of the authors knowledge, has not been discussed in the
literature before. Also the block triangular representation \eqref{eq:triDiag} seems to be new in this general setting.

However, the triangular representation \eqref{eq:triDiag} for closed operators $A+V$ is in general not sufficient to ensure the
spectral identity
\[
 \spec(A+V) = \spec\bigl((A_0+W_1X_0)|_{\dav_0}\bigr) \cup \spec(A_1-X_0W_1)\,.
\]
The reason is that the invariance of $\cG(\cH_0,X_0)$ for the inverse $(A+V-\lambda)^{-1}$ depends on $\lambda$, and it may happen
that $\cG(\cH_0,X_0)$ is invariant only for some $\lambda$ but not necessarily for all, even if $\Dom(A+V)$ is invariant for $X_0$.
This is illustrated in the following example in the setting of bounded block matrices $B=A+V$:

\begin{remark}
 For $\cH_0=\cH_1=\ell^2(\N)$, the Hilbert space of square-summable sequences, let $A_0$ be the right-shift on $\ell^2(\N)$ and
 $A_1$ its adjoint. Choose $W_1$ to map the one-dimensional kernel of $A_1$ onto the orthogonal complement of $\Ran A_0$ with
 trivial continuation to a bounded operator $W_1\colon \cH_1\to\cH_0$. Finally, let $W_0=X=0$, so that on $\cH=\cH_0\oplus\cH_1$ we
 consider the bounded block triangular matrix
 \[
  B = \begin{pmatrix} A_0 & W_1\\ 0 & A_1 \end{pmatrix}.
 \]

 It is easy and straightforward to show that the operator matrix $B$ is bijective, although the diagonal blocks $A_0$ and $A_1$ are
 not. In particular, $\cH_0=\cG(\cH_0,0)$ is invariant for $B$ but not for the inverse $B^{-1}$. However, since $A_0$ and $A_1$ are
 bounded, there is some constant $\lambda$ such that $A_0-\lambda$ and $A_1-\lambda$ are bijective, and so is $B-\lambda$. In this
 case, $\cH_0$ is invariant for both $B-\lambda$ and the inverse $(B-\lambda)^{-1}$.
\end{remark}

Matters change as soon as the invariant graph $\cG(\cH_0,X_0)$ has some nice complementary subspace $\cV$ in the sense that for
\emph{some} $\lambda$ such that $A+V-\lambda$ is bijective both subspaces $\cG(\cH_0,X_0)$ and $\cV$ are invariant for $A+V$ and
$(A+V-\lambda)^{-1}$. Indeed, in this case, the pair $\cG(\cH_0,X_0)$ and $\cV$ decomposes the operator $A+V$ by Lemma
\ref{lem:resInv}. In turn, the graph is invariant for \emph{every} inverse $(A+V-\lambda)^{-1}$, independent of $\lambda$. If, in
addition, $A+V$ is closed, the triangular representation \eqref{eq:triDiag} then implies that $(A_0+W_1X_0)|_{\dav_0}$ and
$A_1-X_0W_1$ are closed as well and the spectral identity
\[
 \spec(A+V) = \spec\bigl((A_0+W_1X_0)|_{\dav_0}\bigr) \cup \spec(A_1-X_0W_1)
\]
takes place (cf.\ \eqref{eq:specid1} and \eqref{eq:specid2}).

In this observation, $\cV$ may be a graph with respect to $\cH_1$ as in Theorem \ref{thm:main} or, like $\cG(\cH_0,X_0)$ is, a
graph with respect to $\cH_0$ as well. The latter case has been discussed to some extend in \cite{TW14}. However, the above
reasoning does not require the subspace $\cV$ to be a graph at all, which makes this situation much more general.

\section{Relatively bounded perturbations}\label{sec:app}

In this section, we discuss a priori assumptions on the diagonal operator $A$ and the off-diagonal perturbation $V$ which guarantee
that the resolvent sets of $A+V$ and $A-YV$ are not disjoint and, therefore, Theorem \ref{thm:main} can be applied. Here, we
consider only the case where $A$ is closed and $V$ is relatively bounded with respect to $A$. We also briefly discuss how these
assumptions are related to those in the previous works \cite{LT97,TW14,W11}, see Remark \ref{rem:single} below.

Recall that a linear operator $H\colon\cH\supset\Dom(H)\to\cH$ is said to be \emph{$A$-bounded} (or \emph{relatively bounded with
respect to $A$}) if  $\Dom(H)\supset\Dom(A)$ and there exist constants $a,b\ge0$ such that
\begin{equation}\label{eq:defRelBound}
 \norm{Hx}\le a\norm{x} + b\norm{Ax} \quad\text{ for }\quad x\in\Dom(A)\,.
\end{equation}
If $H$ is $A$-bounded, the \emph{$A$-bound} of $H$ (or \emph{relative bound of $H$ with respect to $A$}) is defined as the infimum
of all possible choices for $b$ in \eqref{eq:defRelBound}, see \cite[Section IV.1.1]{Kato66}.

A particular case of $A$-bounded operators are linear operators $H\colon\cH\supset\Dom(H)\to\cH$ with $\Dom(H)\supset\Dom(A)$ such
that the operator $H(A-\lambda)^{-1}$ is bounded on $\cH$ for some $\lambda\in\rho(A)$, the resolvent set of $A$. Indeed, in this
case one has (see, e.g., \cite[Lemma 7.2\,(i)]{W11})
\[
 \norm{Hx} \le \norm{H(A-\lambda)^{-1}}\,\norm{(A-\lambda)x}
 \le \norm{H(A-\lambda)^{-1}}\;\abs{\lambda}\;\norm{x} + \norm{H(A-\lambda)^{-1}}\;\norm{Ax}
\]
for all $x\in\Dom(A)$. In particular, the $A$-bound of $H$ does not exceed $\norm{H(A-\lambda)^{-1}}$.

The following criterion, which guarantees that the resolvent sets of $A+V$ and $A-YV$ intersect, follows from a basic Neumann
series argument.

\begin{lemma}\label{lem:Neumann}
 In addition to Hypothesis \ref{hyp:blockMatrix}, assume that $A$ is closed with non-empty resolvent set and that $V$ is
 $A$-bounded. Moreover, suppose that the graph subspaces $\cG(\cH_0,X_0)$ and $\cG(\cH_1,X_1)$ associated with bounded linear
 operators $X_0\colon\cH_0\to\cH_1$ and $X_1\colon\cH_1\to\cH_0$, respectively, are complementary invariant subspaces for $A+V$.
 Define the operator $Y$ as in \eqref{eq:defY}.

 If
 \begin{equation}\label{eq:Neumann}
  \max\{1,\norm{Y}\}\,\cdot\, \norm{V(A-\lambda)^{-1}} < 1 \quad\text{for some}\quad \lambda\in\rho(A)\,,
 \end{equation}
 then the operators $A+V$ and $A-YV$ are closed and $\lambda$ belongs to the intersection of their resolvent sets.

 In particular, the resolvent sets of $A+V$ and $A-YV$ intersect if
 \begin{equation}\label{eq:resolventBound}
  \liminf_{\lambda\in\rho(A)}\norm{V(A-\lambda)^{-1}} = 0\,.
 \end{equation}

 \begin{proof}
  Let $H$ be either $V$ or $-YV$. Write
  \[
   A+H-\lambda = \bigl(I_\cH + H(A-\lambda)^{-1}\bigr)(A-\lambda)\,.
  \]
  By \eqref{eq:Neumann}, $\|H(A-\lambda)^{-1}\|<1$, so that the operator $I_\cH + H(A-\lambda)^{-1}$ has a bounded inverse.
  Therefore, $\lambda$ belongs to the resolvent set of $A+H$. This proves the claim.
 \end{proof}%
\end{lemma}

The following corollary to Lemma \ref{lem:Neumann} ensures a nontrivial intersection of the resolvent sets of $A+V$ and $A-YV$
under conditions in terms of the $A$-bound of $V$ and certain additional spectral properties of $A$. In this context, recall that
the operator $A$ is said to be \emph{$m$-accretive} if all $\lambda$ satisfying $\Re\lambda<0$ belong to the resolvent set of $A$
with $\norm{(A-\lambda)^{-1}}\le \abs{\Re\lambda}^{-1}$, see \cite[Section V.3.10]{Kato66}.

\begin{corollary}\label{cor:Neumann}
 Assume the hypotheses of Lemma \ref{lem:Neumann}. If either
 \begin{enumerate}
  \renewcommand{\theenumi}{\alph{enumi}}
  \item there is a sequence $(\lambda_k)$ in the resolvent set of $A$ such that
        \[
         \abs{\lambda_k}\to\infty\quad\text{as}\quad k\to\infty \quad\text{ and }\quad
         \sup_k\abs{\lambda_k}\,\norm{(A-\lambda_k)^{-1}} < \infty
        \]
        and $V$ has $A$-bound $0$

  or

  \item $A$ is self-adjoint or $m$-accretive and the $A$-bound $b_*$ of $V$ satisfies
        \begin{equation}\label{eq:ABound}
         \max\{1,\norm{Y}\}b_* < 1\,,
        \end{equation}
 \end{enumerate}
 then the operators $A+V$ and $A-YV$ have a common point in their resolvent sets.

 \begin{proof}
  It suffices to show that \eqref{eq:resolventBound} and \eqref{eq:Neumann}, respectively, are satisfied. In case of (a), this
  follows from \cite[Lemma 7.2\,(iii)]{W11}.

  In case of (b), note that for self-adjoint or $m$-accretive $A$ one has
  \[
   \liminf_{\lambda\in\rho(A)} \norm{V(A-\lambda)^{-1}} = b_*\,,
  \]
  see \cite[Lemma 7.2\,(ii)]{W11}; cf.\ also \cite[Satz 9.1]{Wei00}. Thus, condition \eqref{eq:ABound} immediately implies
  \eqref{eq:Neumann}.
 \end{proof}%
\end{corollary}

\begin{remark}
 If the operator $Y$ is in advance known to be a contraction, that is, $\norm{Y}\le 1$, then the conditions \eqref{eq:Neumann} and
 \eqref{eq:ABound} simplify as $\norm{V(A-\lambda)^{-1}}<1$ for some $\lambda\in\rho(A)$ and $b_*<1$, respectively. This situation
 is revisited in Theorem \ref{thm:subordinated} below. Moreover, if $Y$ even is a uniform contraction, that is, $\norm{Y}<1$, then
 the hypothesis on the graphs to be complementary is redundant (see Remark \ref{rem:automorphism}), which further simplifies the
 situation. We refer to \cite{AL95,ADL01,AMM03,GKMV13,KMM2004,KMM2005,LT97,MoSe06}, where different aspects of existence for graphs
 with contractive angular operators, that is, for contractive solutions to operator Riccati equations (cf.\ Lemma
 \ref{lem:invariance}) have been discussed.

 However, if $\norm{Y}>1$, then, in the framework of our approach, we have to stick to conditions \eqref{eq:Neumann},
 \eqref{eq:resolventBound}, or their particular cases in Corollary \ref{cor:Neumann}, that is, roughly speaking, the norm of
 $V(A-\lambda)^{-1}$ has to compensate by being sufficiently small.
\end{remark}

\begin{remark}\label{rem:single}
 Analogous considerations can be made in the (more general) context of single graph subspaces in Theorem \ref{thm:single}. For
 instance, condition \eqref{eq:resolventBound} also guarantees that the operators $A+V-\lambda$ and $A_1-\lambda-X_0W_1$ are
 bijective for some $\lambda$ in the resolvent set of $A$. Since this requires only mild and obvious modifications of the above
 reasoning, we omit the details. However, under the additional assumption \eqref{eq:resolventBound}, one implication of the
 statement of Theorem \ref{thm:single} has already been proved in \cite[Proposition 7.5]{W11} using a different technique;
 cf.\ also \cite[Lemma 6.1]{TW14} and the proof of \cite[Theorem 4.1]{LT97} in particular situations where
 \eqref{eq:resolventBound} is satisfied.
\end{remark}

As a consequence of Lemma 5.1 and its particular cases in Corollary \ref{cor:Neumann}, Theorem \ref{thm:main} can be applied in
these situations. Namely, the pair of subspaces $\cG(\cH_0,X_0)$ and $\cG(\cH_1,X_1)$ decomposes the operator $A+V$ if and only if
the operator $Y$ is a strong solution to the operator Riccati equation $AY-YA-YVY + V=0$, that is,
\[
 \Ran (Y|_{\Dom(A)})\subset\Dom(A)=\dav
\]
and
\begin{equation}\label{eq:opeq0}
 AYx-YAx-YVYx + Vx=0 \quad\text{for}\quad x\in\Dom(A)\,,
\end{equation}
cf.\ Remark \ref{rem:strong}. Moreover, if this equivalence takes place, the operator $A+V$ admits the block diagonalization
\eqref{eq:blockDiag2ndMain}, that is,
\begin{equation}\label{eq:blockDiag2nd}
 (I_\cH-Y)(A+V)(I_\cH-Y)^{-1} = A-YV = \begin{pmatrix} A_0 - X_1W_0 & 0\\ 0 & A_1 - X_0W_1\end{pmatrix}.
\end{equation}

It has already been mentioned in Remark \ref{rem:alternative} that in the current case of $\Dom(A)\subset\Dom(V)$ the block
diagonalization \eqref{eq:blockDiag2nd} can be rewritten as the block diagonalization
\begin{equation}\label{eq:blockDiag1st}
 (I_\cH+Y)^{-1}(A+V)(I_\cH+Y) = A+VY = \begin{pmatrix} A_0 + W_1X_0 & 0\\ 0 & A_1 + W_0X_1\end{pmatrix},
\end{equation}
where the latter has been the object of focused attention in the literature so far. In particular, Theorem \ref{thm:main} in
combination with Lemma \ref{lem:Neumann} and Corollary \ref{cor:Neumann} can be considered a direct extension of
\cite[Lemma 5.3 and Theorem 5.5]{AMM03}, where the diagonal part $A$ is assumed to be self-adjoint and the off-diagonal
perturbation $V$ is assumed to be a bounded symmetric operator. Here, we also have the additional (new) block diagonalization
formula \eqref{eq:blockDiag2nd}. At the same time, our argument closes a gap in reasoning in the proof of \cite[Lemma 5.3]{AMM03},
where it has implicitly been assumed that $I_\cH\pm Y$ maps $\Dom(A)$ onto itself.

\section{An example for self-adjoint $2\times 2$ operator matrices}\label{sec:Example}

The general conditions \eqref{eq:Neumann} or \eqref{eq:resolventBound} in Section \ref{sec:app} may not be sufficient to guarantee
the required existence of invariant graph subspaces for $A+V$ in Theorem \ref{thm:main}. Stronger conditions on the diagonal part
$A$ and/or the off-diagonal perturbation $V$ that guarantee the existence of such graph subspaces have been discussed, e.g., in
\cite{LT97,TW14,W11}. In those situations, Theorem \ref{thm:main} applies automatically. 

In this last section, we revisit the case of self-adjoint $2\times 2$ operator matrices where the spectra of the diagonal entries
are subordinated but may have a one point intersection. This situation has previously been discussed in \cite{KMM2004,LT97} and
complements a generalization of the well-known Davis-Kahan $\tan2\Theta$ theorem.

Recall that a closed subspace $\cU$ is said to \emph{reduce} a linear operator $B$ if the pair $\cU$ and its orthogonal complement
$\cU^\perp$ decompose $B$. In particular, $\cU$ reduces $B$ if and only if $\cU^\perp$ does. In this context, it is worth to note
that the orthogonal complement of a graph subspace $\cG(\cH_0,X)$ with some linear operator $X\colon\cH_0\to\cH_1$ is again a
graph, namely
\[
 \cG(\cH_0,X)^\perp = \cG(\cH_1,-X^*)\,.
\]
Also recall that for a self-adjoint operator $B$ with spectral measure $\EE_B$, every spectral subspace $\Ran\EE_B(\Delta)$ with a
Borel set $\Delta\subset\R$ automatically reduces $B$, see, e.g., \cite[Satz 8.15]{Wei00}.

As an application of our abstract results we have the following statement, part (a) of which is essentially known
under even more general assumptions, see Theorem 2.7.7 and the extension mentioned in Remark 2.7.12 and Proposition 2.7.13 in
\cite{Tre08}. Part (b) of our theorem, however, strengthens the corresponding statements of Theorem 2.7.21, Corollary 2.7.23, and
Theorem 2.8.5 in \cite{Tre08}.

\begin{theorem}\label{thm:subordinated}
 In addition to Hypothesis \ref{hyp:blockMatrix}, assume that the diagonal operator $A$ is self-adjoint and that the spectra of
 $A_0$ and $A_1$ are subordinated in the sense that
 \[
  \sup\spec(A_0) \le \mu \le \inf\spec(A_1) \quad\text{for some}\quad \mu\in\R\,.
 \]
 Furthermore, suppose that the off-diagonal perturbation $V$ is symmetric with $W_0=W_1^*$.

 \begin{enumerate}
  \item If $\Dom(V)\supset \Dom(A)$ and the operator matrix $B=A+V$ is self-adjoint on its natural domain $\Dom(B)=\Dom(A)$, then
        the subspace
        \begin{equation}\label{eq:defL}
         \cL:=\Ran\EE_{B}\bigl((-\infty,\mu)\bigr) \oplus \bigl(\Ker(B-\mu)\cap\cH_0\bigr)
        \end{equation}
        reduces the operator $B$. Moreover, this subspace can be represented as a graph
        \begin{equation}\label{eq:graphL}
         \cL=\cG(\cH_0,X)
        \end{equation}
        with some linear contraction $X\colon\cH_0\to\cH_1$.

  \item If even $V$ is $A$-bounded with $A$-bound smaller than $1$, then the skew-symmetric operator
        $Y=\begin{pmatrix} 0 & -X^*\\ X & 0 \end{pmatrix}$ with $X$ as in \eqref{eq:graphL} is a strong solution to the operator
        Riccati equation $AY-YA-YVY+V=0$. Moreover, the operator $B=A+V$ admits the block diagonalizations \eqref{eq:blockDiag2nd}
        and \eqref{eq:blockDiag1st} with $W_0,X_0,X_1$ replaced by $W_1^*,X,-X^*$, respectively. In particular, the operators
        $A+VY$ and $A-YV$ are mutually adjoint to each other on their natural domain $\Dom(A)$.
 \end{enumerate}
\end{theorem}

Clearly, the subspace $\cL$ in \eqref{eq:defL} satisfies
\[
 \EE_B\bigl((-\infty,\mu)\bigr) \subset \cL \subset \EE_B\bigl((-\infty,\mu]\bigr)
\]
and either inclusion may a priori be strict. In particular, $\cL$ is not necessarily a spectral subspace for $B$. On the other
hand, if $\mu$ is not an eigenvalue of $B$ (cf.\ Remark \ref{rem:spectra} below), then $\cL$ is spectral with
\begin{equation}\label{eq:cLLT}
 \cL = \Ran\EE_{A+V}\bigl((-\infty,\mu)\bigr) = \Ran\EE_{A+V}\bigl((-\infty,\mu]\bigr)\,.
\end{equation}
In this particular case, the graph representation \eqref{eq:graphL} for $\cL$ has already been shown explicity in
\cite[Theorem 3.1]{LT97}; see also \cite[Theorem 2.7.7]{Tre08}.

\begin{remark}\label{rem:comparisonLT}
 In the situation of \eqref{eq:cLLT}, a direct application of \cite[Corollary 3.2]{LT97} only yields
 \[
  AYx - YAx - YVYx + Vx = 0
 \]
 and the block `diagonalization'
 \[
  (I_\cH+Y)^{-1}(A+V)(I_\cH+Y)x=(A+VY)x
 \]
 for $x\in \cD=\Ran\bigl((I_\cH+Y)^{-1}|_{\Dom(A)}\bigr)$, which is a somewhat weaker statement than the one of Theorem
 \ref{thm:subordinated}\,(b), unless it is known in advance that $\cD=\Dom(A+VY)=\Dom(A)$. In our line of reasoning, the latter
 domain equality is a byproduct of the block diagonalization \eqref{eq:blockDiag2nd}, see also Remark \ref{rem:alternative}.

\end{remark}

The main part of the proof of Theorem \ref{thm:subordinated} consists in establishing the representation \eqref{eq:graphL} under
the sole assumption that $B=A+V$ is self-adjoint. As in \cite[Theorem 2.4]{KMM2004}, where the bounded case was discussed, this is
done by reducing the problem to the case where $\mu$ is not an eigenvalue of $B$. To this end,
it is crucial to know that the kernel of $B-\mu$
splits with respect to the decomposition $\cH=\cH_0\oplus\cH_1$, that is,
\begin{equation}\label{eq:kernelSplit}
 \Ker(B-\mu) = \bigl(\Ker(B-\mu)\cap \cH_0\bigr)\oplus \bigl(\Ker (B-\mu)\cap \cH_1\bigr)\,,
\end{equation}
cf.\ \cite[Remark 3.4]{LT97} and also \cite[Remark 2.7.12]{Tre08}.
This information is provided by the following lemma, which is already known explicitly in the bounded case (see
\cite[Theorem 2.2]{KMM2004}) and for some unbounded operator matrices under similar assumptions, see \cite[Proposition 2.7.13]{Tre08};
cf.\ \cite[Theorem 2.13]{Schm15} for an analogue for form perturbations.

\begin{lemma}
 Under the hypotheses of Theorem \ref{thm:subordinated}, the kernel of $B-\mu$ splits as
 \begin{equation}\label{eq:kernel}
  \Ker(B-\mu) = \bigl(\Ker(A_0-\mu)\cap\Ker W_1^*\bigr) \oplus \bigl(\Ker(A_1-\mu)\cap\Ker W_1\bigr)\subset \Ker(A-\mu)\,.
 \end{equation}
 In particular, the identity \eqref{eq:kernelSplit} holds.

 \begin{proof}
  Clearly, it suffices to show \eqref{eq:kernel}, the proof of which is essentially the same as for the bounded case in
  \cite{KMM2004}. For the sake of completeness, we reproduce it here.
  
  Let $x=f\oplus g\in\Ker(B-\mu)$ with $f\in\dav_0$ and $g\in\dav_1$ be arbitrary. Then
  \begin{equation}\label{eq:kerApply}
   (A_0-\mu)f + W_1g = 0 \quad \text{ and } \quad W_1^*f + (A_1-\mu)g=0 \,.
  \end{equation}

  Assume that $f\notin\Ker(A_0-\mu)$. It then follows from \eqref{eq:kerApply} that
  \[
   \langle f,W_1g\rangle = -\langle f,(A_0-\mu)f\rangle > 0
  \]
  and, therefore,
  \[
   \langle g,(A_1-\mu)g\rangle = -\langle g,W_1^*f\rangle = -\langle W_1g,f\rangle < 0\,,
  \]
  which is a contradiction to the hypothesis $A_1\ge \mu$. Hence, $f\in\Ker(A_0-\mu)$ and, in turn, $g\in\Ker W_1$ by
  \eqref{eq:kerApply}. Analogously, it follows that $g\in\Ker(A_1-\mu)$ and $f\in\Ker W_1^*$. This shows that the left-hand side of
  \eqref{eq:kernel} is contained in the right-hand side. The converse inclusion is obvious, so that \eqref{eq:kernel} holds.
 \end{proof}%
\end{lemma}  

We are now ready to turn to the proof of Theorem \ref{thm:subordinated}.

\begin{proof}[Proof of Theorem \ref{thm:subordinated}]
 (a). {\it Step 1.} The spectral subspace $\Ran\EE_{B}(\{\mu\})=\Ker(B-\mu)$ for $B$ reduces also the operator $A$.

 Indeed, observe that $\Ran\EE_{B}(\{\mu\})$ is by \eqref{eq:kernel} a subspace of $\Ker(A-\mu)$, so that $\Ran\EE_{B}(\{\mu\})$
 is invariant for $A$. In view of $\Ran\EE_{B}(\{\mu\})\subset\Dom(A)$, it then follows by standard reasoning that the orthogonal
 complement $\Ran\EE_{B}(\{\mu\})^\perp$ is invariant for $A$ as well. Finally, taking into account that $\Dom(A)=\Dom(B)$ and that
 $\Ran\EE_{B}(\{\mu\})$ reduces $B$, one has
 \[
  \Dom(A)=\bigl(\Dom(A)\cap\Ran\EE_{B}(\{\mu\})\bigr) + \bigl(\Dom(A)\cap\Ran\EE_{B}(\{\mu\})^\perp\bigr)\,,
 \]
 so that $\Ran\EE_{B}(\{\mu\})$ reduces $A$.

 {\it Step 2.}  The subspace $\cL$ reduces the operator $B$:

 The pair $\Ran\EE_{B}\bigl((-\infty,\mu)\bigr)$ and
 $\Ran\EE_{B}\bigl([\mu,\infty)\bigr)=\Ran\EE_{B}(\{\mu\})\oplus\Ran\EE_{B}\bigl((\mu,\infty)\bigr)$ clearly decomposes the
 operator $B$. Taking into account the kernel splitting \eqref{eq:kernel}, it is then easy to see that the pair $\cL$ and
 \[
  \cL^\perp = \Ran\EE_{B}\bigl((\mu,\infty)\bigr) \oplus \bigl(\Ran\EE_{B}(\{\mu\})\cap\cH_1\bigr)
 \]
 decomposes $B$ as well, that is, $\cL$ reduces $B$.

 {\it Step 3}. Finally, we prove the graph representation \eqref{eq:graphL} by reducing the problem to the case where $\mu$ is not
 an eigenvalue of $B$. The corresponding reduction process is essentially the same as in Step 2 of the proof of
 \cite[Theorem 2.4]{KMM2004}, where the bounded case was discussed.

 Taking into account that by Step 1 the subspace $\hat{\cH}:=\Ran\EE_{B}(\{\mu\})^\perp$ reduces both $A$ and $B$, denote by
 $\hat{A}:=A|_{\hat{\cH}}$ and $\hat{B}:=B|_{\hat{\cH}}$ the corresponding parts of $A$ and $B$, respectively. In particular,
 $\hat{A}$ and $\hat{B}$ are self-adjoint, and $\Ker(\hat{B}-\mu)=\{0\}$.

 In view of the kernel splitting \eqref{eq:kernel}, a simple standard reasoning shows that also the orthogonal complement $\hat{H}$
 splits with respect to $\cH=\cH_0\oplus\cH_1$, that is,
 \[
  \hat{\cH}=\hat{\cH}_0\oplus\hat{\cH}_1\quad  \text{with}\quad 
  \hat{\cH}_0:=\cH_0\cap\hat{\cH} \quad\text{and}\quad \hat{\cH}_1:=\cH_1\cap\hat{\cH}\,.
 \]
 With respect to this orthogonal decomposition, the operator $\hat{A}$ has a diagonal representation
 \[
  \hat{A} = \begin{pmatrix} \hat{A}_0 & 0\\ 0 & \hat{A}_1 \end{pmatrix},
 \]
 and the self-adjoint operator $\hat{B}$ is an off-diagonal perturbation of $\hat{A}$, more precisely
 \[
  \hat{B} = \begin{pmatrix} \hat{A}_0 & 0\\ 0 & \hat{A}_1 \end{pmatrix} + \begin{pmatrix} 0 & \hat{W}\\ \hat{W}^* & 0 \end{pmatrix}
  = \begin{pmatrix} \hat{A}_0 & \hat{W}\\ \hat{W}^* & \hat{A}_1 \end{pmatrix},\quad \hat{W}:=W_1|_{\hat{\cH}_1}\,.
 \]
 
 Clearly, one has
 \[
  \sup\spec(\hat{A}_0) \le \mu \le \inf\spec(\hat{A}_1) \quad\text{and}\quad \Ker(\hat{B}-\mu)=\{0\}\,.
 \]
 Without loss of generality we may assume that the entry $\hat W$ is closed. Indeed, since $\hat{B}$ is self-adjoint,
 the entry $\hat{W}^*$ is densely defined, so that $\hat{W}^{**}\supset\hat{W}$. Taking into account the domain inclusion
 $\Dom(A)\subset\Dom(V)$, and hence $\Dom(\hat{A}_1)\subset\Dom(\hat{W})$, we may therefore replace $\hat{W}$ with $\hat{W}^{**}$
 to obtain
 \[
  \hat{B}=\begin{pmatrix} \hat{A}_0 & \hat{W}^{**}\\ \hat{W}^* & \hat{A}_1 \end{pmatrix}.
 \]
  Now, it  follows from \cite[Theorem 3.1]{LT97} that the subspace
 $\hat{\cL}:=\Ran\EE_{\hat{B}}\bigl((-\infty,\mu]\bigr)\subset\hat{\cH}$ is a graph with respect to a linear contraction
 $\hat{X}\colon\hat{\cH}_0\to\hat{\cH}_1$, that is, $\hat{\cL} = \cG(\hat{\cH}_0,\hat{X})$.
 
 Naturally embedding $\hat{\cH}$ into $\cH$, one has
 \begin{equation}\label{eq:H0}
  \cH_0 = \hat{\cH}_0 \oplus \bigl(\EE_{B}(\{\mu\})\cap\cH_0\bigr)
 \end{equation}
 and $\Ran\EE_{\hat{B}}\bigl((-\infty,\mu]\bigr)=\Ran\EE_{B}\bigl((-\infty,\mu)\bigr)$, so that
 \begin{equation}\label{eq:graph}
  \begin{aligned}
   \cL &= \hat{\cL} \oplus \bigl(\Ran\EE_{B}(\{\mu\})\cap\cH_0\bigr)\\
       &= \cG(\hat{\cH}_0,\hat{X}) \oplus \bigl(\Ran\EE_{B}(\{\mu\})\cap\cH_0\bigr)\,.
  \end{aligned}
 \end{equation}
 Define $X\colon\cH_0\to\cH_1$ by
 \begin{equation}\label{eq:defX}
  Xf := \begin{cases}
         \hat{X}f\,, & f\in\hat{\cH}_0\\
         0\,, & f\in\EE_{B}(\{\mu\})\cap\cH_0\,.
        \end{cases}
 \end{equation}
 The representation \eqref{eq:graphL} then follows from \eqref{eq:H0}--\eqref{eq:defX}. Moreover, since $\hat{X}$ is a contraction,
 $X$ is a contraction as well. This completes the proof of part (a).

 (b). First, note that $B$ automatically is self-adjoint if $V$ is $A$-bounded with $A$-bound smaller than $1$, see, e.g.,
 \cite[Theorem V.4.3]{Kato66}. By part (a), the subspace $\cL$ in \eqref{eq:defL} therefore is a graph $\cL=\cG(\cH_0,X)$ with some
 linear contraction $X\colon\cH_0\to \cH_1$. Taking into account that $\cG(\cH_0,X)^\perp=\cG(\cH_1,-X^*)$, the operator $Y$ then
 agrees with \eqref{eq:defY} for the pair of graphs $\cG(\cH_0,X)$ and $\cG(\cH_1,-X^*)$. It clearly is a contraction along with
 $X$, and, therefore, condition \eqref{eq:ABound} in Corollary \ref{cor:Neumann} holds. Thus, the operators $B=A+V$ and $A-YV$ have
 a common point in their resolvent sets.

 Since the subspace $\cL$ reduces $B$, the Riccati equation and the two block diagonalizations now follow from equations
 \eqref{eq:opeq0}--\eqref{eq:blockDiag1st} in the previous section. In turn, the remaining statement that $A-YV$ and $A+VY$ are
 mutually adjoint to each other follows from \eqref{eq:blockDiag2nd} and \eqref{eq:blockDiag1st}. Indeed, by
 \cite[Satz 2.43]{Wei00}, the adjoint of the left-hand side of \eqref{eq:blockDiag2nd} agrees with the one of
 \eqref{eq:blockDiag1st} and vice versa. This completes the proof.
\end{proof}%

It is interesting to note that standard theory such as \cite[Propositions 1.6 and 1.7]{Schm12} only yields the two extension
relations $(A+VY)^*\supset A-YV$ and $(A-YV)^*\supset A+VY$, and it requires additional efforts to show that the operator
equalities hold. In our situation, this is guaranteed by the availability of the two block diagonalizations \eqref{eq:blockDiag2nd}
and \eqref{eq:blockDiag1st}.

\begin{remark}\label{rem:spectra}
 Taking into account the kernel splitting \eqref{eq:kernel}, the subspace $\cL$ in \eqref{eq:defL} is a spectral subspace for
 $B=A+V$ if and only if $\Ker(A_0-\mu)\cap\Ker W_1^*$ or $\Ker(A_1-\mu)\cap\Ker W_1$ is trivial, in which case
 \[
  \cL = \EE_B\bigl((-\infty,\mu)\bigr) \quad\text{or}\quad \cL = \EE_B\bigl((-\infty,\mu]\bigr)\,,
 \]
 respectively. In particular, this takes place if $\Ker(A-\mu)$ is trivial, and then by \eqref{eq:kernel} also $\Ker(B-\mu)$ is
 trivial; cf.\ the discussion after Theorem \ref{thm:subordinated} above.

 If the spectrum of $A$ even satisfies $\sup\spec(A_0) < \inf\spec(A_1)$, then it is well known that the gap in the spectrum of $A$
 persists in the spectrum of $B=A+V$ and that the angular operator $X$ from the graph representation $\cL=\cG(\cH_0, X)$ is a
 uniform contraction, that is, $\norm{X}<1$, see \cite[Theorem 2.1]{AL95} and \cite[Theorem 3.1]{LT97}; a more precise bound on
 $\norm{X}$ for this case is provided by \cite[Theorem 1]{MoSe06} (see also \cite[Theorem 3.1]{GKMV13}).
\end{remark}

We close the section with an example that illustrates the application of Theorem \ref{thm:subordinated} in an
off-diagonally dominant situation. It concerns the
diagonalization of a Hamiltonian that describes massless Dirac fermions in the presence of an impurity in graphene
\cite{GN07,NGMJKGDF05,DM84}; cf.\ \cite[Theorem 3.3.7]{Tre08} for a diagonalization of the massive 
three-dimensional
Dirac operator with electromagnetic field.

\begin{example}\label{exgraph} In a zero-gap semiconductor, low energy electrons moving in the vicinity of an impurity can formally
 be described by the two-dimensional Dirac-like Hamiltonian
 \[
  H=\frac{\hbar \nu_F}{i} {\boldsymbol \sigma} \cdot {\boldsymbol\nabla}+U\,,
 \]
 where $\boldsymbol\sigma =(\sigma_x,\sigma_y)$, with $\sigma_x,\sigma_y$ the $2\times 2$ conventional Pauli matrices
 \[
  \sigma_x=\begin{pmatrix}0&1\\1&0\end{pmatrix},\quad\sigma_y=\begin{pmatrix}0&-i\\i&0\end{pmatrix},
 \]
 $\nu_F$ is the Fermi velocity and $U $ is a short range ``defect'' potential
\cite{DM84}.

 It is well known (see, e.g., \cite{Th92}) that the standard \emph{Foldy-Wouthuysen transformation} $\cT_{\textrm{FW}}$ on
 $L^2(\R^2;\C^2)$, given in the momentum representation by the multiplication operator by the unitary matrix
 \[
  \frac{1}{\sqrt{2}}
  \begin{pmatrix} \theta({\bf k}) & 1\\ \theta({\bf k}) & -1\end{pmatrix}
  \quad\text{with}\quad 
  \theta({\bf k})=\frac{\sqrt{k_x^2+k_y^2}}{k_x-ik_y}\,,
 \]
 diagonalizes the free Dirac operator $H_0:=\frac{\hbar \nu_F}{i} {\boldsymbol \sigma} \cdot {\boldsymbol \nabla}$. More precisely,
 in the system of units where $\hbar\nu_F=1$, one has
 \[
  \cT_{\textrm{FW}} H_0 \cT_{\textrm{FW}}^{-1} = \begin{pmatrix} \sqrt{-\Delta} & 0\\ 0 & -\sqrt{-\Delta} \end{pmatrix}
 \]
 with respect to the decomposition
 \[
  L^2(\R^2;\C^2)=L^2(\R^2)\oplus L^2(\R^2)\,.
 \]

 Now, it is easy to see that
 \begin{equation}\label{eq:DiracDiag}
  \cT_{\textrm{FW}} H \cT_{\textrm{FW}}^{-1} =
  \begin{pmatrix}
   \sqrt{-\Delta}+U+\Theta U \Theta^*& -U+\Theta U \Theta^*\\
   -U+\Theta U \Theta^*&-\sqrt{-\Delta}+U+\Theta U \Theta^*
  \end{pmatrix},
 \end{equation}
 where $\Theta $ is the  Fourier multiplier with the unimodular symbol $\theta({\bf k})$.

 As is known, the massless pseudorelativistic Hamiltonian $\sqrt{-\Delta}$ is subcritical in dimension two and higher:
 $\sqrt{-\Delta}+V$ does  not have a bound state if the negative part of the potential $V$ is not ``deep enough" (see
 \cite[Theorem V.1]{CMS90}; also see \cite[Eq.\ (2.15)]{Dau83} and \cite{LS97}). For instance, if $U$ is a compactly supported
bounded potential with $\|U\|_\infty$ small enough, the decomposition 
 \begin{equation}\label{eq:diagSplit}
  \pm\sqrt{-\Delta}+U+\Theta U\Theta^*=\frac{1}{2}\left((\pm\sqrt{-\Delta} +2U) + \Theta(\pm\sqrt{-\Delta} +2U)\Theta^*\right)
 \end{equation}
 shows that the diagonal entries of \eqref{eq:DiracDiag} have no negative/positive spectrum, respectively. In this case, Theorem
 \ref{thm:subordinated} applies, which means that the operator matrix \eqref{eq:DiracDiag} can be block diagonalized. In turn, the
 operator $H$ can be block diagonalized with respect to the decomposition
 \[
  L^2(\R^2; \C^2)=\cH_+\oplus\cH_-\,,
 \]
 where $\cH_\pm=\Ran\left(\EE_{H_0}(\R_\pm)\right)$ denote the ``electronic"/``positronic" subspaces of the free (massless)
 Dirac operator $H_0$, respectively.
\end{example}

\end{document}